\documentclass[12pt]{amsart}
  
\input xy
\xyoption{all}
  
\usepackage{amssymb}
\usepackage{amsmath}
\usepackage{amscd}
\usepackage{latexsym}
\usepackage{color}
 
\begin{document}
 
\newtheorem{lemma}{Lemma}[section]
\newtheorem{prop}[lemma]{Proposition}
\newtheorem{cor}[lemma]{Corollary}
  
\newtheorem{thm}[lemma]{Theorem}
\newtheorem{Mthm}[lemma]{Main Theorem}
\newtheorem{con}{Conjecture}
\newtheorem{claim}{Claim}
\newtheorem{ques}{Question}

\theoremstyle{definition}
  
\newtheorem{rem}[lemma]{Remark}
\newtheorem{rems}[lemma]{Remarks}
\newtheorem{defi}[lemma]{Definition}
\newtheorem{ex}[lemma]{Example}
                                                                                
\newcommand{\C}{\mathbb C}
\newcommand{\R}{\mathbb R}
\newcommand{\Q}{\mathbb Q}
\newcommand{\Z}{\mathbb Z}
\newcommand{\N}{\mathbb N}

\title[Cryptosystems using subgroup distortion]{Cryptosystems using subgroup distortion}

\author[I. Chatterji]{Indira Chatterji}
\address{Indira Chatterji,  Laboratoire J.A. Dieudonn\'e de l'Universit\'e de Nice, France}
\email{indira.chatterji@math.cnrs.fr}

\author[D. Kahrobaei]{Delaram Kahrobaei}
\address{Delaram Kahrobaei, CUNY Graduate Center, PhD Program in Computer Science and NYCCT, Mathematics Department, City University of New York, New York University, Department of Computer Science and Engineering}
\email{dkahrobaei@gc.cuny.edu}

\author[N. Y. Lu]{Ni Yen Lu}
\address{Ni Yen Lu, CUNY Graduate Center, City University of New York}
\email{nlu@gradcenter.cuny.edu}

\maketitle
\begin{abstract}
In this paper we propose cryptosystems based on subgroup distortion in hyperbolic groups. We also include concrete examples of hyperbolic groups as possible platforms.
\end{abstract}
\tableofcontents
\section{Introduction}
Using algorithmic problems in non-commutative groups for cryptography is a fairly new but very active field for over a decade (see for instance \cite{MSU2}). In this paper we propose new cryptosystems using subgroup distortion. The algorithmic problems which are proposed for non-commutative group-based cryptography so far are: {\it Conjugacy Search Problem, Endomorphism Search Problem, Word Choice problem}, {\it Membership search problem} and {\it Twisted Conjugacy Problem} among others. There has not been yet any proposal to use the {\it Geodesic Length Problem} or {\it Complexity of Distortion in Subgroups} as we do in this paper. We propose a couple of symmetric cryptosystems based on these problems, and analyze their security.

The paper is organized as follows: in Section \ref{disto} we discuss the notion of subgroup distortion and in Section \ref{gl} we discuss the problem of finding the geodesic length of an element in a group in polynomial time, and explain how in a Gromov hyperbolic group this can be done in polynomial time. In Section \ref{crypto} we explain two possible protocols based on subgroup distortion, and in Section \ref{cex} we give a few concrete examples of hyperbolic groups that can be used as platforms for the cryptosystems described in Section \ref{crypto}.
\section{Basics group theory facts}
\subsection{Subgroup distortion}\label{disto} Let $G$ be a finitely generated group and $S\subseteq G$ a finite generating set. Then for $g\in G$ the {\it word length associated to }$S$ is given by
$$\ell_S(g)=\min\{n\in\N | g=s_1\dots s_n, s_i\in S\cup S^{-1}\}$$
For any two finite generating sets $S,S'$ of $G$, there is a constant $C\geq 1$ such that, for any $g\in G$ one has
$$\ell_S(g)\leq C\ell_{S'}(g).$$
For $H<G$ a finitely generated subgroup, if $T\subset H$ is a generating set, then for any $h\in H$
$$\ell_{S\cup T}(h)\leq\ell_T(h).$$
Indeed, there are ``shortcuts" to the identity when one is allowed to use both elements from the generating set from $G$ and $H$. Those shortcuts may no longer be there when we are restricted to the generating set of $H$ and hence the other inequality is in general not true. By how much this other inequality fails is how one defines the distortion. In the rest of the paper we will assume that $T\subseteq S$, so that $S\cup T=S$.
\begin{defi}
Let $G$ be a finitely generated group and $H<G$ be a finitely generated subgroup. The {\it distortion} of $H$ in $G$ is a function
\begin{eqnarray*}{\tt Dist}^G_H:\N&\to&\N\\
n&\mapsto&\max\{\ell_T(h)\,|\,\ell_S(h)\leq n\}\end{eqnarray*}
Notice that a priori this function depends on the generating sets $S$ and $T$ for $G$ and $H$, but two finite generating sets will give equivalent distortion functions, that is functions that differ by multiplicative constants.
\end{defi}
The following are very natural examples of finitely generated groups with distorted subgroups.
\begin{ex}The metabelian {\it Baumslag-Solitar group}:
$$G=BS(1,2)= \langle a,b\,|\,aba^{-1}=b^2 \rangle$$
If we take $H\simeq\Z=\langle b \rangle$, then one checks that for any $n\in\N$ one has
\begin{eqnarray*}a^nba^{-n}&=&a^{n-1}b^2a^{-(n-1)}=a^{n-1}ba^{-1}aba^{-(n-1)}=a^{n-2}b^4a^{-(n-2)}\\
&=&\dots=b^{2^n}\end{eqnarray*}
Hence $\ell_{\{b\}}(b^{2^n})=2^n$ whereas $\ell_{\{a,b\}}(b^{2^n})=2n+1$ so that ${\tt Dist}^G_H$ is at least an exponential.
\end{ex}
\begin{ex}The integer {\it Heisenberg group}, given by
$$G=H_\Z= \langle a,b,c\,|\,[a,c]=[b,c]=e, [a,b]=c \rangle$$
If we take $H\simeq\Z= \langle c \rangle$, this is the center of $G$ and then one checks that for any $n\in\N$, using that $ab=cba$ and that $a^{-1}b^{-1}=cb^{-1}a^{-1}$ one has
\begin{eqnarray*}a^nb^na^{-n}b^{-n}&=&a^{n-1}abb^{n-1}a^{-n}b^{-n}=a^{n-1}cbab^{n-1}a^{-n}b^{-n}\\
&=&ca^{n-1}bab^{n-1}a^{-n}b^{-n}=ca^{n-1}b^2cab^{n-2}a^{-n}b^{-n}\\
&=&c^2a^{n-1}b^2ab^{n-2}a^{-n}b^{-n}=\dots=c^na^{n-1}b^na^{-n+1}b^{-n}\\
&=&\dots=c^{2n}a^{n-2}b^na^{-n+2}b^{-n}=\dots=c^{n^2}b^nb^{-n}=c^{n^2}\end{eqnarray*}
And hence $\ell_{\{c\}}(c^{n^2})=n^2$ whereas $\ell_{\{a,b,c\}}(c^{n^2})=4n$ so that ${\tt Dist}^G_H$ is at least a polynomial.
\end{ex}
\subsection{The Geodesic Length Problem}\label{gl}
Given a finitely generated group $G$ and a finite generating set $S$, one can ask the following.
\begin{ques}[Geodesic Length Problem]
What is the complexity of the algorithm that given $g\in G$, finds the geodesic length $\ell_S(g)$?
\end{ques}
This question seems hard in general, and not much studied. In \cite{MRUV} it is shown that this problem is NP-complete in the free-metabelian group $S_{r,2}$. It is also known that in free groups or Right Angled Artin groups given by standard generating sets, there are fast algorithms for computing the geodesic length of elements \cite{MSU2}. In braid groups, or nilpotent groups, the computation of the geodesic length of elements is hard \cite{MSU2}.\\

There are many groups of exponential growth where the Geodesic Problem is decidable in polynomial time, for example, hyperbolic groups \cite{Ep} or metabelian Baumslag-Solitar group BS(1,n), \cite{Elder}. Notice that, a priori, the Geodesic Problem is a bit harder than the Geodesic Length Problem: indeed, once one has found a geodesic, one automatically has its length, but knowing the length of a geodesic doesn't give the geodesic. However, according to Elder and Rechnitzer in \cite{ER}, those two problems are polynomially reducible to each other, meaning that a polynomial time solution to one of the problems is equivalent to a polynomial time solution to the other one.

\bigskip

In the case where $G$ is hyperbolic in the sense of Gromov, the following is easy:
\begin{thm}[Epstein et al \cite{Ep}]\label{pol}Let $G$ be a Gromov hyperbolic group, then the Geodesic Problem (hence the Geodesic Length Problem as well) is solvable in polynomial time.\end{thm}
\begin{proof} According to \cite{BH} (Part III.$\Gamma$.2), in a Gromov hyperbolic group a word has a normal form which is a quasi-geodesic, so one can check by hand in a neighborhood of this quasi-geodesic to find the geodesic length, and the neighborhood of a quasi-geodesic is the same as a neighborhood of a geodesic according to Morse lemma.\end{proof}
Tim Riley's alternative argument is that for hyperbolic groups, one can transform an arbitrary path into a local (for the right definition of local) geodesic in polynomial time and then conclude it is a geodesic.
\subsection{The membership search problem}\label{member} Given $G$ a finitely generated group, with a finite generating set $S$ and $H$ a subgroup with its own generating set $T$, one can ask the following.
\begin{ques}[Membership search problem] Given $h\in H$ expressed in terms of the generating set $S$, how long is it needed to express $h$ in terms of elements of $T$?\end{ques}
The difficulty of the membership search problem has been used in cryptography by Shpilrain and Zapata in \cite{SZ}, but here we will be needing examples in which the membership search problem is polynomial, see Lemma \ref{pallavi}.
\section{The Cryptosystems Using Subgroup Distortion}\label{crypto}
\subsection{The protocol I: basic idea} Assume that Alice and Bob would like to communicate over an insecure channel. Here $G=\langle g_1, \cdots ,g_l |R \rangle$ is a public group and $H= \langle t_1,\cdots, t_s \rangle \subset\langle g_1, \cdots ,g_l  \rangle =G$ is a secret subgroup of $G$, that is distorted and shared between only between Alice and Bob. We further assume that the geodesic length problem is polynomial both in $G$ and $H$, and that the membership search problem is polynomial in $H$. Then: 
\begin{enumerate}
	\item Alice picks $h\in H$ with $\ell_H(h)=n$, expresses $h$ in terms of generators of $G$ with $\ell_G(h)=m \ll n$ and sends $h$ to Bob.
	\item Bob then converts $h$ back  in terms of generators of $H$ and computes $\ell_H(g)=n$ in polynomial time to recover $n$.
\end{enumerate}
\subsubsection{Security} Although $H$ is not known to anyone except to Alice and Bob and $h$ being sent with length $m\ll n$ gives infinitely many possible guesses for the eavesdropper Eve, the security of the scheme is weak since Eve will have intercepted enough elements of $H$ to generate $H$ (one can think of the group ${\bf Z}$ of the integers, it is enough to intercept two relatively prime integers to generate the whole group). 

\subsection{The protocol I: secure version}
We suggest making it impossible for Eve to tell which elements in the sent form belong to $H$ by sending along $h$ several elements that do not belong to $H$. To determine how Bob can tell which elements belong to $H$ to retrieve the correct message we will consider below the subgroup membership problem and the random number generator.
\subsubsection{Subgroup membership problem}
Suppose we have a group in which the subgroup membership problem is solved efficiently then we will send some random words and the receiver first checks whether each word belongs to $H$ and then computes its length.\\
Protocol:\\
Let $G = \langle g_1,\cdots,g_l \vert R_G\rangle$ be a group that is known to the public and $H= \langle h_1,\cdots,h_s \rangle $ be a secret subgroup of $G$ that is exponentially distorted. Assume that the subgroup membership problem in $G$ efficiently solvable, and that as in Protocol I the word problem is polynomial in $G$, the geodesic length problem and the membership problems are both polynomial in $H$. Then:
\begin{enumerate}
	\item Alice picks $h\in H$ with $\ell_H(h)=n$, expresses $h=g_1 \cdots g_{m}$ in terms of generators of $G$ with $\ell_G(h)=m \ll n$. She randomly generates $a_0,\dots,a_m \in G \setminus H$  and sends these words to Bob.
	\item Since Bob knows the generating set for $H$, he find $h \in H$ (since he could check the subgroup membership problem efficiently) he only uses $h\in H$, in terms of generators of $H$ and computes $\ell_H(g)=n$ in polynomial time according to our assumptions to recover $n$.
\end{enumerate}
\subsubsection{Random Number Generator} 
Suppose we have a random number generator and two parties that share the same random number generator and the same seed, they will get a same random sequence. We would like to use this idea but instead on groups.\\
This notion is possible if we are given a one-to-one correspondence between the countably infinite set of integers and a countably infinite group $G$. There is a natural ordering of elements in the group of integers and so we can impose this ordering on $G$. Generating $m$ random numbers is the same as generating $m$ elements in group $G$.\\
The advantage here is that given the same random number generator and the same seed, two parties would produce the same sequence of random numbers and by using the "same ordering" in $G$, they would get the same sequence of random elements of $G$.\\
According to Cayley's theorem, every group is isomorphic to a permutation group. There is a lexicographic ordering on permutation groups and hence there is a unique ordering of any group. There is a one-to-one correspondence between the group of integers and a countably infinitely group. 
We will use the idea of random number generator for group in the protocol below.\\
Protocol:\\
Let $G=D_1*D_2* \cdots *D_n$ where each $D_i$ is a Gromov hyperbolic group that is known to the public. Alice and Bob share $H= \langle d_,\cdots,d_s \rangle \subset D_1:=\langle d_1, \cdots ,d_l |R_1 \rangle $ which is an exponentially distorted hyperbolic subgroup of $D_1$ (and hence $G$), a random number generator, and a way to choose a seed. (For example, they could use the date and time for the seed : 02032016123342 where 02-03-2016 is today date and 12:33:42pm is the current time of message being sent. They could also add to this the number sent by previous message.) 
\begin{enumerate}
	\item Alice picks $h\in H$ with $\ell_H(h)=n$, expresses $h=d_1 \cdots d_{m}$ in terms of generators of $D_1$ with $\ell_G(h)=m \ll n$. She randomly generates a sequence of $(m+1)$ numbers from the random number generator and picks $a_0,\dots,a_m$ that belong to $D_2* \cdots *D_n$ that is in a one-to-one correspondence with the sequence of $(m+1)$ numbers.\footnote{According to Cayley's theorem, every group is isomorphic to a permutation group so we can use the lexicographic ordering of permutation group and then order the product $D_2*\cdots*D_m$ lexicographically.} She then sends $a_0d_1a_1d_2\dots d_ma_m$ to Bob (the $a_i$'s are expressed in a fixed generating set for $D_2* \cdots *D_n$).
	\item Bob knows the random number generator and the seed so he knows which $a_{i}$'s are sent along with $h$. He uses $a_i$'s inverses to get back $h=d_1 \cdots d_{m}$. Since he also knows $H$, he converts $h$ back in terms of generators of $H$ and computes $\ell_H(g)=n$ in polynomial time according to Theorem \ref{pol} to recover $n$.
\end{enumerate}

\subsubsection{Security} The security of the scheme relies on the fact that:
\begin{itemize}
	\item $H < G$ is not known to anyone except to Alice and Bob. 
	\item Since $h$ is sent with length $\ell_G(h)=m \ll n$, there are infinitely many guesses for Eve that are greater than $m$.
	\item For both protocols, only Bob can tell which elements sent in the form of $h$ belong to $H$. For the second protocol, the random number generator and the seed are known to only Alice and Bob, so there is no way for Eve to tell which elements among $\{h,a_i\}$ belong to $H$ to try to generate $H$.
\end{itemize}

\subsection{The protocol II: basic idea}
Let $G =\langle S|R_1\rangle $ be a secret group that is only known only to Alice and Bob and that has polynomial geodesic length problem.  Let $H =\langle T \rangle $ be a  public distorted subgroup of $G$. Here $T$ is a subset of $S$.
\begin{enumerate}
	\item Alice wants to send a message $n\in \mathbb{N} $ to Bob. She picks $g\in G$ with $\ell_G(g)=n$. She then expresses $g =t_1 t_2 t_3 \cdots  t_m$, where $m \gg n$ and $t_i's \in T$ and sends to Bob.
	\item Bob converts $g$ back in terms of generators of $G$ and by assumption computes its length in polynomial time to recover $n$.
\end{enumerate}

\subsubsection{Security}
Although $G$ is not known to anyone except Alice and Bob, the security of this scheme is not strong since the eavesdropper could potentially guess the value of $n$ based on the upper bound $m$.

\subsection{The protocol II: secure version} Instead of sending $g$ with $\ell_G(g)=m \gg n$ we can embed $H$ exp(exp) distorted in another group $K$ so that we can transmit message of size $\leq \log n < m$. 
For the cryptosystem below, we need the following groups:
$$G=\langle {g_1}, \cdots ,g_l |R_G \rangle$$
$$H= \langle h_1, h_2, \cdots, h_k |R_H \rangle $$
$$K=\langle {k_1}, \cdots ,k_q |R_K \rangle$$
where $H$ is a distorted subgroup of $G$ and embedded exp(exp) distorted in $K$. The group $H$ is known to the public whereas $G$ and $K$ are known to only Alice and Bob.
\begin{enumerate}
	\item Alice wants to send a message $n\in \mathbb{N} $ to Bob.
	 She picks $g\in G$ with $\ell_G(g)=n$, $g=g_1g_2 \cdots g_n$.\\
	 Since $H$ is distorted in $G$, there is $m > n$ with $\ell_H(g)=m$. Alice then expresses $g =h_1 h_2\dots h_m$, in terms of generators of $H$.\\
	 Since $H$ is embedded exp(exp) distorted in $K$, there exist $p \ll \ll m$ and $k_1, k_2, \cdots k_p$ in the generating set of $K$ such that $g = k_1k_2 \cdots k_p$.\\
	 Alice sends $g$ in this form to Bob.
	\item Bob will do the following:\\
	 He uses his knowledge of $K$ and $H$ and the fact $H$ is exp(exp) in $K$ to convert $g = k_1k_2 \cdots k_p (p \ll m)$ to $g = h_1 h_2\dots h_m$.\\
	 Since he knows that $H$ is distorted in $G$, he converts $g =h_1 h_2\dots h_m$ to $g= g_1 g_2\dots g_n$ back in terms of generators of $G$.
	 He then computes the length of $h$ to recover $n$.
\end{enumerate}

\subsubsection{Security} The security of the scheme relies on the fact that finding the geodesic length problem in $H$ for the eavesdropper is impossible due to the fact that:
\begin{itemize}
	\item $G$ and $K$ are not known to anyone except to Alice and Bob. 
	\item $g$ is sent in terms of generators of $K$ so there is no way for Eve to figure out $H$.
	\item With $\ell_K(g)=p \ll n$, there are infinitely many choices of numbers greater than $p$ for guessing.
\end{itemize}
\section{Possible platforms}\label{cex}
For both protocols, Gromov hyperbolic groups seem to provide interesting platforms. Indeed, according to Theorem \ref{pol} the geodesic length problem is solvable in polynomial time.  There are many examples of hyperbolic groups with exponentially distorted hyperbolic subgroups, see for instance \cite{MJ} for geometric examples such as surface subgroups in fundamental groups of hyperbolic 3-manifolds, but we do not know about membership search problems there.
\subsection{Free-by-cyclic platforms for protocol I} One possible weakness of Protocol I is that the public group $G$ does not contain enough exponentially distorted subgroups $H$, so Eve could make a group theoretic search and find all the distorted subgroups. To avoid that problem, one could use hyperbolic groups which can be written as free-by-cyclic groups in infinitely many ways.  Such groups are constructed in~\cite{MM}.  More precisely, the authors construct groups $G$ which are hyperbolic, and have infinitely many homomorphisms to $\Z$, with free kernel.  Given any such homomorphism, one has an expression 
$G= F_n(a_1, \dots, a_n) \rtimes_\phi \langle t \rangle$, where the first factor is the free group on the generators $a_1, \dots, a_n$, and $\phi$ is an automorphism of $ F_n(a_1, \dots, a_n)$ such that 
$t a_i t^{-1} = \phi(a_i)$ for all $i$.  

We fix one such $G= F_n(a_1, \dots, a_n) \rtimes_\phi \langle t \rangle$ (including a choice of generators $a_1, \dots, a_n, t$) as the public group $G$.  

Now Alice and Bob together choose one of the (infinitely many) other homomorphisms of $G$ to $\Z$, say $G= F_m (b_1, \dots, b_m) \rtimes \langle s \rangle$ and take $H = F_m(b_1,\dots, b_m) <G$.  
\begin{lemma}\label{pallavi}The membership search problem for $H<G$ is solvable in polynomial time. \end{lemma}
\begin{proof}
Given a word $w = w(a_1, \dots, a_n, t)$ a word in the public generators $a_1, \dots, a_n, t$ of $G$, which represents an element $h\in H$, we need to show that there is a polynomial time algorithm to write $h$ in terms of the generators 
$b_1, \dots, b_m$ of $H$. 

Since $G= F_n(a_1, \dots, a_n) \rtimes_\phi \langle t \rangle = F_m (b_1, \dots, b_m) \rtimes \langle s \rangle$, each $a_i$ can be written as a word in $b_1, \dots, b_m, s$.  Thus by hyperbolicity, $w$ can be changed into a word $v=v(b_1, \dots, b_m, s)$ in real time, and 
there is a constant $K$, depending only on $H$, such that $|v| \le K|w|$.  

The word $v$ may have some powers of $s$ and $s^{-1}$, but since it represents the element $h$ of $H$, it has an expression $u$ which is a word in just $b_1, \dots b_m$.  Applying 
Britton's lemma to $u^{-1}v$, we see that $v$ must have an innermost $s, s^{-1}$ pair: i.e., $v$ must have a subword  of the form $sxs^{-1}$ or $s^{-1}x s$, where $x$ is a word in  just $b_1, \dots b_m$.  
Replace this subword with $\phi(x)$ or $\phi^{-1}(x)$ respectively, to get a word $v_1$ representing $h$
with fewer $s$'s and $s^{-1}$'s than $v$.  Continuing this procedure, after finitely many  steps  
we will have written down an expression for $h$ in terms of $b_1, \dots, b_m$.  Applying the automorphism involves multiplication in the group, which is polynomial.  Moreover, the 
number of steps is bounded above by the number of $s, s^{-1}$ pairs, which is 
at most $|v|/2 < K|w|$.  
\end{proof}
\subsection{Exponential and Exp(exp) distortion for protocol II}We now provide concrete examples of hyperbolic groups with an exponentially and an exp(exp) distorted subgroup that could be used in protocol II (improved version). Those examples are a particular case of the more general techniques developed in \cite{BBD}. Here we describe a specific type which may fit our needs, although it is not clear that they have a fast enough membership search problem.\\

Let $$G_1:=\langle a_1, a_2, \cdots , a_{14}, t_1 | t_1^{-1}a_jt_1=w_{1j}   (1\le j \le 14) \rangle$$ 
and 
$$G_{14}:= \langle  a_1, a_2,  \cdots , a_{14^2}, t_1,  \cdots , t_{14} | t_i^{-1}a_j t_i = w_{ij}   (1 \le i \le 14, 1\le j \le 14^2) \rangle$$
where $w_{1j}$'s are positive words on $a_j$'s, of length 14 such that $a_i a_j$ appears at most once as a subword of $w_{1j}$ and similarly for $w_{ij}$. We obtain $w_{1j}$ by noting that the following word 
$$(a_1a_1a_2a_1a_3a_1 \cdots a_{14})(a_2a_2a_3a_2  \cdots a_{14}) \cdots (a_{13}a_{13}a_{14})a_{14}$$ 
has length $14^2$ so we can split it into $14$ subwords of length $14$, each corresponding to $w_{1j}$.
\subsection{Exponentially Distorted Subgroups}
The subgroup $$F_1:= \langle a_1, \cdots, a_{14}\rangle $$ is free of rank $14$ and is exponentially distorted in $G_1$. 

Here is an example. The word $t_1^{-n}a_1t_1^n$ has length $2n+1$ in $G_1$. On the other hand, 

\begin{eqnarray*}
t_1^{-n}a_1t_1^n &=& t_1^{-n+1}t_1^{-1}a_1t_1t_1^{n-1}
            = t_1^{-n+1}w_{11}t_1^{n-1} \\
                 &=& t_1^{-n+1}a_1a_1a_2a_1  \cdots a_7a_1 t_1^{n-1} \\
&=& t_1^{-n+2}t_1^{-1}a_1t_1t_1^{-1}a_2 \cdots t_1^{-1}a_7t_1t_1^{-1}a_1t_1t_1^{n-2} \\
&=& t_1^{-n+2}w_{11}w_{11}w_{12} \cdots w_{17}w_{11}t_1^{n-2} \\
&=&  \cdots a^{(1)}_j \cdots
\end{eqnarray*}

Since $l_{G_1}(t_1^{-n}a_1t_1^n)=2n+1 $ and $l_{F_1}(t_1^{-n}a_1t_1^n)= 14^n$, the subgroup $F_1$ is at least exponentially distorted in $G_1$. 

\subsection{Exponentially Exponentially Distorted Subgroups}
Define $H:=G_1 *_{F_1} G_{14}$. 
Denote $$F_{1}:= \langle  a^{(1)}_1,  \cdots, a^{(1)}_{14}\rangle$$ 
and 
$$F_{2}:= \langle a^{(2)}_1,  \cdots , a^{(2)}_{14}, \cdots , a^{(2)}_{14^2} \rangle.$$

Let $w_{1}=t^{-n}a^{(1)}_1t^n$ and
\begin{eqnarray*}
w_{2}&=&w_{1}^{-1}a^{(2)}_1w_{1} \text{ where } {a_1^{(2)} \in F_2} \text{ and } {w_1 \in F_1}\\
	&=&(t^{-n}a^{(1)}_1t^n)^{-1}a^{(2)}_1t^{-n}a^{(1)}t^n \in H\\
	&=&(\cdots a^{(1)}_j \cdots)^{-1}a^{(2)}_1(\cdots a^{(1)}_j \cdots)\\
	&=&	(\cdots t_j \cdots)^{-1}a^{(2)}_1(\cdots t_j \cdots)\\
	&=&= \cdots a^{(2)}_k \cdots
\end{eqnarray*} 
where the third equality follows by the previous computation and the fourth equality follows since $F_1$ is identified with the subgroup of $G_{14}$ generated by $\langle  t_1, \cdots, t_{14} \rangle $. There are $14^n$ elements in each of $(\cdots a^{(1)}_j \cdots)$ so $14^n$ $t_j$'s on each side of $a^{(2)}_1$. Since $l_H(w_2)=4n+2 $ and $l_{F_2}(w_2)= 14^{14^n}$, $F_2$ is at least exp(exp) distorted in $H$.
\section*{Acknowledgements} The authors thank Pallavi Dani for conversations on an earlier draft and for pointing out free-by-cyclic groups as a possible platform as well as Lemma \ref{pallavi}.
Indira Chatterji is partially supported by the Institut Universitaire de France (IUF).  
Delaram Kahrobaei is partially supported by a PSC-CUNY grant from the CUNY Research Foundation, the City Tech Foundation, and ONR (Office of Naval Research) grant N00014-15-1-2164. Delaram Kahrobaei has also partially supported by an NSF travel grant CCF-1564968 to IHP in Paris.


\begin{thebibliography}{99}

\bibitem{BH} M. Bridson and A. Haefliger, \emph{Metric Spaces of Non-Positive Curvature.} Springer, Grundlehren der mathematischen Wissenschaften, 1999.

\bibitem{Elder} M. Elder,  \emph{A linear time algorithm to compute geodesics in solvable Baumslag-Solitar groups.} Illinois Journal of Mathematics 54 (2010) Number 1 pages 109--128.

\bibitem{ER}  M. Elder, A. Rechnitzer, \emph{Some geodesic problems for finitely generated groups.} Groups, Complexity, Cryptology 2 (2010) Issue 2, 223--229.

\bibitem{Ep} D. B. A. Epstein, J. W. Cannon, D. F. Holt, S. V. F. Levy, M. S. Paterson, and W. P. Thurston, \emph{Word processing in groups.} Jones and Bartlett Publishers, 1992.


\bibitem{MM} T. Mecham and A. Mukherjee. \emph{Hyperbolic groups which fiber in infinitely many ways.} Algebraic \& Geometric Topology 9 (2009) 2101--2120.


\bibitem{MJ} Mahan Mitra, \emph{Coarse extrinsic geometry: a survey.} Geometry and Topology Monographs Volume 1: The Epstein birthday schrift (1998), 341--364.

\bibitem{MRUV}A. Myasnikov, V. Romankov, A. Ushakov, A.Vershik. \emph{The Word and Geodesic Problems in Free Solvable Groups}, Transactions of the American Mathematical Society 362 (9), (2010), 4655--4682.

\bibitem{BBD} J. Barnard, N. Brady, P. Dani. \emph{Super-Exponential Distortion of Subgroups of CAT(-1) Groups}, Algebraic $\&$ Geometric Topology, 7 (2007), 301--308.

\bibitem{MSU2}
A. G. Myasnikov, V. Shpilrain, and A. Ushakov, {\sl Non-commutative cryptography and complexity of group-theoretic problems}, Amer. Math. Soc. Surveys and Monographs, 2011.
\bibitem{SZ}V. Shpilrain and G. Zapata. \emph{Using the subgroup membership search problem in public key cryptography}. Groups, Complexity, and Cryptology 1 (2009), 33--49.

\end{thebibliography}
\end{document}